\documentclass[twoside,english,review]{elsarticle}
\usepackage[T1]{fontenc}
\usepackage[latin9]{inputenc}
\usepackage{verbatim}
\usepackage{units}
\usepackage{amsmath}
\usepackage{amsthm}
\usepackage{amssymb}
\usepackage{esint}
\usepackage{nomencl}

\providecommand{\makenomenclature}{\makeglossary}
\makenomenclature

\makeatletter
\theoremstyle{plain}
\newtheorem{thm}{\protect\theoremname}
\theoremstyle{plain}

\ifx\proof\undefined
\newenvironment{proof}[1][\protect\proofname]{\par
\normalfont\topsep6\p@\@plus6\p@\relax
\trivlist
\itemindent\parindent
\item[\hskip\labelsep\scshape #1]\ignorespaces
}{%
\endtrivlist\@endpefalse
}
\providecommand{\proofname}{Proof}
\fi
\theoremstyle{remark}
\newtheorem{rem}[thm]{\protect\remarkname}
\theoremstyle{plain}
\newtheorem{cor}[thm]{\protect\corollaryname}
\theoremstyle{definition}
\newtheorem{defn}[thm]{\protect\definitionname}
\theoremstyle{plain}
\newtheorem{prop}[thm]{\protect\propositionname}

\usepackage{caption}

\makeatletter
\def\ps@pprintTitle{%
  \let\@oddhead\@empty
  \let\@evenhead\@empty
  \let\@oddfoot\@empty
  \let\@evenfoot\@oddfoot
}
\makeatother

\makeatother

\usepackage{babel}
\providecommand{\corollaryname}{Corollary}
\providecommand{\definitionname}{Definition}
\providecommand{\lemmaname}{Lemma}
\providecommand{\propositionname}{Proposition}
\providecommand{\remarkname}{Remark}
\providecommand{\theoremname}{Theorem}

\usepackage{xspace}
\usepackage{xcolor}
\usepackage{mathtools}

\newcommand{\ie}{\unskip, i.\,e.,\xspace}
\newcommand{\eg}{\unskip, e.\,g.,\xspace}

\newcommand\blfootnote[1]{%
	\begingroup
	\renewcommand\thefootnote{}\footnote{#1}%
	\addtocounter{footnote}{-1}%
	\endgroup
}

\begin{document}

\begin{frontmatter}{}

\title{Analysis of extremum value theorems for function spaces in optimal control under numerical uncertainty}

\author[TUC]{Pavel~Osinenko}


\author[TUC]{Stefan~Streif}


\address[TUC]{Automatic Control and System Dynamics Lab;
Technische Universit\"at Chemnitz, 09107 Chemnitz, Germany}
\begin{abstract}
The extremum value theorem for function spaces plays the central role in optimal control. It is known that computation of optimal control actions and policies is often prone to numerical errors which may be related to computability issues. The current work addresses a version of the extremum value theorem for function spaces under explicit consideration of numerical uncertainties. It is shown that certain function spaces are bounded in a suitable sense \ie they admit finite approximations up to an arbitrary precision. The proof of this fact is constructive in the sense that it explicitly builds the approximating functions. Consequently, existence of approximate extremal functions is shown. Applicability of the theorem is investigated for finite--horizon optimal control, dynamic programming and adaptive dynamic programming. Some possible computability issues of the extremum value theorem in optimal control are shown on counterexamples.\end{abstract}

\end{frontmatter}{}

\blfootnote{This is an author's original version of an article accepted for publication in IMA Journal of Mathematical Control and Information. The version of record is available online at: \texttt{dx.doi.org/10.1093/imamci/dny018}}

\section{Introduction\label{sec:intro}}

Optimal control represents an important part of control theory. Typically, one seeks for an optimal function over a state space (also called \emph{control policy}) so as to minimize a given cost functional. It is, however, not in general possible to compute optimizing control policies exactly due to limitations of numerical procedures which may have certain effects on the system behavior. The current work shows how, under some mild and practicable assumptions, approximate optimal control policies can still be explicitly computed. The proofs are done constructively \ie they entail certain ways of computing the objects in question. Constructive results are not unusual in control engineering and are often desired: for instance, \citet*{Banaschewski1997-constr-Stone-Weier} gave a constructive proof of the Stone-Weierstrass theorem, which is used in a number of applications to effectively find approximations to specific functions in a suitable basis. The famous Sontag's formula \citep{Sontag1989-formula} is the core of the Sontag's constructive proof of the Artstein's theorem on nonlinear stabilization. \citet*{Sepulchre2012-constr-nonlin-ctrl} developed this methodology further into a vast variety of constructive methods of finding specific stabilizing controllers.

Going back to the problem of optimal control and related effects which may occur due to numerical uncertainty, consider the following simple example of a discrete-time system whose dynamical behavior is switched by a binary decision variable $u$:
\begin{align}
	x_{k+1}= \begin{cases}
		(\frac{1}{2} + b)x_k, &u_k = 1, \\
		(\frac{1}{2} + c)x_k, &u_k = -1,
	\end{cases} \quad x_0 = 1, u_k \in \{1, -1\},
	\label{eq:exm-switched-sys}
\end{align}
where $b, c$ are real numbers which may \eg represent some physical quantities. Let an infinite-horizon cost function be defined as:
\[
\min_{\{u_k\}_k} \quad J = \sum_{k=0}^{\infty} x_k^2.
\]

Suppose, for the sake of the example, that one of the numbers $b, c$ is zero and the other one is positive. Then, by the virtue of the system dynamics \eqref{eq:exm-switched-sys}, the optimal control policy is $u^*=\{1,1,1, \dots \}$ and the corresponding optimal state sequence is $x^*=\{ 1,\frac{1}{2},\frac{1}{4}, \frac{1}{8}, \dots \}$. In this case, the optimal cost is $J^*=2$. Thus, if one could find the optimal control policy, i. e., the optimal control action at each time step, then either $J^*=\frac{2}{1-2c}=2$ or $J^*=\frac{2}{1-2b}=2$ by the geometric sum and hence either $b=0$ or alternatively $c=0$. However, in practice, there may occur a numerical uncertainty between the exact values of $b, c$ and their representations in a computational device, usually as rational numbers. One of the possible simple ways to consider these representations of $b, c$ is in the form of Cauchy sequences $\{b(n)\}_n, \{c(n)\}_n$ which are regular in the following sense:
\begin{align*}
	& \forall n,m \in \mathbb{N} \\ 
	& |b(n)-b(m)| \le \tfrac{1}{n} + \tfrac{1}{m}, \\
	& |c(n)-c(m)| \le \tfrac{1}{n} + \tfrac{1}{m}.
\end{align*}

In practice, the system \eqref{eq:exm-switched-sys} may contain some particular approximations $b(n'), c(n'), n' \in \mathbb N$ where $n'$ is the precision of the computational device.In the current work, all the proofs are done by working directly with the representations of real numbers which helps address numerical uncertainty. The said rational approximations may well come from \eg a measurement, which always has a finite precision, or from some computational algorithm, such as model identification. Therefore, to computationally check whether $b=0$ or alternatively $c=0$, approximations $b(n), c(n)$ for all $n \in \mathbb N$ must be compared. Such an unbounded search is, however, not technically possible. Therefore, different optimal control policies might result depending on precision -- in this case, a particular number $n'$. 

The same issue may appear when minimizing \eg the following particular cost function:
\[
J(u) = \min \{ u^2 + b , (u-1)^2+c \}.
\]

By the virtue of the numbers $b$ and $c$ as described above, it follows that $\min J = 0$. However, if an optimal control action $u^*$ could be computed exactly, such that $J(u^*)= \min J$, then either $u^* = 0$ or $u^* = 1$. It would be equivalent to deciding whether $b$ or $c$ is exactly zero which is not always technically possible. This has been typically demonstrated in simple counter-examples \citep{Bishop1967-constr-analys}, whose more detailed description may be found in Appendix. Particular examples of peculiar phenomena related to numerical uncertainty and floating-point arithmetic may also be found in \citep{Rump2010-float-pt-arithm}.

%
As shown in the example above, optimality in general may fail to be achieved depending on the representation of system parameters. To address these issues, the present work seeks to show existence of optimal control in an approximate format by explicitly considering numerical uncertainty. The proofs are done constructively and in the setting of \citep{Bishop1985-constr-analysis} since it offers convenient tools for keeping track of the number representations. The details are given in the next section.

It should be noted that, classically, the extremum value theorem states the following:

\begin{thm}
	If a function $f$ is continuous on a compact interval $[a,b]$, then there exist $x,y\in[a,b]$ such that $f(x)=\sup f$ and $f(y)=\inf f$.
	\label{thm:EVT}
\end{thm}

Some constructive approaches to Theorem \ref{thm:EVT} were addressed \eg in \citep{Berger2006-EVT,Bridges2007-EVT} with additional assumptions on the function $f$. These assumptions, are, however, not always easy to verify practically, especially when one wants to apply the theorem to function spaces. Instead of strengthening the conditions of the theorem, an approximate format is considered in the present work which is sufficient for practical applications of optimal control. To achieve this, it is shown, that certain function spaces admit finite approximations. The proof is based on constructing finite approximations explicitly. 

Another consequence of the new results shows also what at best can be achieved in general when addressing optimal control. The major implication is a theoretical limit at which any numerical algorithm may perform. Whereas exact optimal control policies are not achievable in general, the new result demonstrates principal possibility of computing approximate optimal control policies up to prescribed accuracy provided that the optimization problem satisfies certain conditions which are, as will be shown in the case study of Section \ref{sec:case-study}, practicable. The next section discusses the important preliminaries needed to prove the main theorem of Section \ref{sec:main-result}.


\section{Preliminaries\label{sec:prelim}}

In this section, the definitions and some basic technical results necessary for derivation of the new approximate extremum value theorem for function spaces are recalled. For a comprehensive description, refer,
for example, to \citet{Bishop1985-constr-analysis,Bridges1987-varieties,Bridges2007-techniques,Ye2011-SF,Schwichtenberg2012-constr-analys}. A \textbf{real} number $x$ in the current work well be characterized by its rational approximations in the following \emph{regular} Cauchy sequence format: 
\[
\forall n,m\in\mathbb{N},|x(n)-x(m)|\leq\frac{1}{n}+\frac{1}{m},
\]
where $x(n)$ is some operation that produces the $n-$th rational approximation to $x$. The inequalities on real numbers are defined as follows:
\begin{align*}
	(x\le y)\triangleq & \forall n\in\mathbb{N},x(n)\le y(n)+\frac{2}{n},\\
	(x<y)\triangleq & \exists n\in\mathbb{N},x(n)<y(n)-\frac{2}{n}.
\end{align*}

In the second definition, the number $n$ is also called a \emph{witness}. Such objects are said to certify the respective formulas. They can be used by computational devices. Further, the maximum of two real numbers is defined as follows: $\max\left\{ x,y\right\} (n)\triangleq\max\left\{ x(n),y(n)\right\}$. The basic properties of it can be proven, but in general it cannot be decided whether $\max\left\{ x,y\right\} =x$ or $\max\left\{ x,y\right\} =y$. However, the following simple technical result can be easily proven:

\begin{prop}
	For any two real numbers $x,y$ satisfying $x\le y$, it follows that $\max\left\{ x,y\right\} = y$.
	\label{lem:max-2-reals}
\end{prop}
\begin{proof}	
	It suffices to show that 
	\[
	\forall n,|\max\left\{ x(n),y(n)\right\} -y(n)|\leq\frac{2}{n}.
	\]
	
	It follows that $\forall n\in\mathbb{N},\max\left\{ x(n),y(n)\right\} \le y(n)+\frac{2}{n}$
	from the condition of the proposition which implies $\forall n\in\mathbb{N}.x(n)\le y(n)+\frac{2}{n}$.
	On the other hand, $\forall n\in\mathbb{N},\max\left\{ x(n),y(n)\right\} \ge y(n)\ge y(n)-\frac{2}{n}$,
	and the result follows.
\end{proof}

\begin{rem}
	Properties of the minimum are derived similarly. 
\end{rem}

A metric space $\left(X,\rho\right)$ is a set $X$ together with an operation $\rho:X\times X\rightarrow\mathbb{R}$ that satisfies the usual axioms of a metric. A metric space $\left(X,\rho\right)$ is \textbf{totally bounded} if for all natural $k$, there exists a finite set of unequal points $\left\{ x_{1}, \dots, x_{n}\right\} \subset X$ such that for any $x\in X$, there exists an $x_{i}\in\left\{ x_{1}, \dots, ,x_{n}\right\}$ with $\rho\left(x,x_{i}\right)\leq\frac{1}{k}$. 
Such a finite set is also called a $\frac{1}{k}-$\textbf{approximation} to $X$. A subset $A$ of a metric space $\left(X,\rho\right)$ is \textbf{located} if it is non-empty and for any $x$ in $X$, the metric $\rho\left(x,A\right)\triangleq\inf\left\{ \rho(x,y):y\in A\right\}$
can be effectively computed. A totally bounded subset of a metric space is also located (see
Proposition 2.2.9 in \citep{Bridges2007-techniques}). A metric between two subsets $A$ and $B$ is defined as $\rho\left(A,B\right)\triangleq\inf\left\{ \rho(x,y):x\in A,y\in B\right\} $.

A (uniformly continuous) \textbf{function} from a totally bounded
metric space $\left(X,\rho\right)$ to a metric space $\left(Y,\sigma\right)$
is a pair consisting of an operation $x\mapsto f(x),x\in X$ and an
operation $\omega:\mathbb{Q}\rightarrow\mathbb{Q}$ called \textbf{modulus}
\textbf{of} (uniform) \textbf{continuity} such that:
\[
\forall\varepsilon\in\mathbb{Q},\forall x,y\in X,\rho(x,y)\leq\omega(\varepsilon)\implies\sigma(f(x),f(y))\leq\varepsilon.
\]

A function is \textbf{Lipschitz continuous} if $\forall x,y\in X,\sigma(f(x),f(y))\leq L\cdot\rho(x,y)$ for some rational $L>0$. The set $\mathcal{F}$ of (all) uniformly continuous functions from a totally bounded metric space $\left(X,\rho\right)$ to a metric space $\left(Y,\sigma\right)$ together with the metric $\tau(f,g)\triangleq\underset{x\in X}{\sup}\sigma(f(x),f(y))$ for any $f,g\in\mathcal{F}$ is called the \textbf{function space} from $X$ to $Y$.
A function space \textbf{$\mathcal{F}$} is \textbf{equicontinuous} if there exists a common modulus of continuity for all $f$ in $\mathcal{F}$.
Further, $\mathcal{F}$ is the space of uniformly Lipschitz and uniformly bounded functions whenever there exists a common Lipschitz constant $L$ and respectively a common bound $K\in\mathbb{Q},K>0$ such that $\forall f\in\mathcal{F},\|f\|\le K$. A uniformly continuous \textbf{functional} $F$ on a totally bounded function space $\mathcal{F}$ is an operation $f\mapsto F[f]\in\mathbb{R}$ with a modulus of continuity $\alpha$ such that $|F[f]-F[g]|\leq\frac{1}{k}$ whenever $\tau(f,g)\leq\alpha\left(\frac{1}{k}\right)$.

The symbol $x^{i}$ denotes the $i-$th coordinate of the point $x$ in $\mathbb{R}^{n}$. The two common norms on $\mathbb{R}^{n}$ are the $d_{2}$--norm: $\|x\|_{2}=\left(\sum_{i=1}^{n}\left(x^{i}\right)^{2}\right)^{\nicefrac{1}{2}}$,
and the $d_{\infty}$--norm (or maximum norm): $\|x\|_{\infty}=\max_{i}\big|x^{i}\big|$.
The subscripts ``$2$'' and ``$\infty$'' may be omitted whenever
the type of the norm is clear from the context. The corresponding
metric between any two points $x,y$ is defined as $\|x-y\|$. For the metrics $d_{2}$ and $d_{\infty}$, the following holds: $\|\bullet\|_{\infty}\leq\|\bullet\|_{2}\leq\sqrt{n}\|\bullet\|_{\infty}$.
A real space $\mathbb{R}^{n}$ with the metric $d_{\infty}$
will be also denoted as $\left(\mathbb{R}^{n},d_{\infty}\right)$.
A (rational) closed \textbf{ball} $\bar{\mathcal{B}}(b,K)$ in $\mathbb{R}^{n}$
with a radius $K\in\mathbb{Q},K>0$ centered at $b\in\mathbb{Q}^{n}$
is the set $\left\{ x:x\in\mathbb{R}^{n}\land\|x-b\|\leq K\right\} $.
For example, with the $d_{\infty}$--metric, $\bar{\mathcal{B}}(b,K)$
is effectively a hypercube with $2^{n}$ vertices with rational coordinates.
On the reals $\mathbb{R}$, it is a compact interval. Clearly, a closed ball $\bar{\mathcal{B}}(b,K)$ is located. A \textbf{regular partition} with a step $\delta=\frac{K}{k},k\in\mathbb{N}$
on a closed ball $\bar{\mathcal{B}}(0,K)$ in $\mathbb{R}^{n}$ is
a finite set of points $\left\{ b_{i}\right\} _{i=1}^{N}\subset\bar{\mathcal{B}}(0,K),N\in\mathbb{N}$
with the coordinates satisfying all the combinations of the form $b_{i}^{j}:=\pm n_{ij}\delta,n_{ij}\in\left\{ 0, \dots, k\right\} $.
Notice that the number $N$ of partition points depends on the dimension
$n$ and the step $\delta$. For instance, a regular partition on
$\bar{\mathcal{B}}\left(\frac{1}{2},\frac{1}{2}\right)$ in $\mathbb{R}$
-- which is the unit interval $\left[0,1\right]$ -- with a step $\delta=\frac{1}{4}$
is the finite set of points $\left\{ 0,\frac{1}{4},\frac{1}{2},\frac{3}{4},1\right\} $.
Clearly, regular partitions on nontrivial closed balls exist and they may be considered as witnesses for total boundedness.
That is, for any approximation to a closed ball, there exists a regular partition with the same property that for
any point in the set there exists a point in the partition that is
close to the given point up to the given precision. The following simple technical result can be easily proven:
\begin{prop}
	Let $P=\left(p_{1}, \dots, ,p_{N}\right)$
	be a regular partition with a step $\delta=\frac{K}{k},k\in\mathbb{N}$
	on a closed ball $\bar{\mathcal{B}}(0,K)\subset\left(\mathbb{R}^{n},d_{\infty}\right)$.
	Then, for any $x \in \bar{\mathcal{B}}(0,K)$, there exists a closed ball $\bar{\mathcal{B}}\left(p_{i},\delta\right),p_{i}\in P$
	such that $x\in\bar{\mathcal{B}}\left(p_{i},\delta\right)$.
	\label{lem:detection-in-hypercubes}
\end{prop}

\begin{proof}
	Since $x$ is a tuple of $n$ real numbers $\left(x^{1}, \dots, x^{n}\right)$,
	an $m-$th rational approximation to $x$ is a tuple of rational numbers
	$x(m):=\left(x^{1}(m), \dots, x^{n}(m)\right)$. Indeed, for any $m'\in\mathbb{N}$,
	$\|x(m)-x(m')\|$$=\max_{i}\big|x^{i}(m)-x^{i}(m')\big|\leq\frac{1}{m}+\frac{1}{m'}$
	since $\big|x^{i}(m)-x^{i}(m')\big|\leq\frac{1}{m}+\frac{1}{m'}$
	for all $i\in\left\{ 1, \dots, n\right\}$. Let $m:=\lceil\frac{4}{\delta}\rceil$,
	then $\|x(m)-x\|\leq\frac{\delta}{2}$. Compute all distances $\big|\big|x(m)-p_{i}\big|\big|,p_{i}\in B$.
	If $\big|\big|x(m)-p_{i}\big|\big|<\frac{\delta}{2}$, then $x\in\bar{\mathcal{B}}\left(p_{i},\delta\right)$.
	If there are more than one such balls, pick the one with the smallest
	index. If $\big|\big|x(m)-p_{j}\big|\big|=\frac{\delta}{2}$ for some
	indices $j=j_{1}, \dots, j_{L}$, then pick the smallest such $j$ and
	conclude that $x\in\bar{\mathcal{B}}\left(p_{j},\delta\right)$.
\end{proof}

\begin{rem}
	In the current constructive setting, it cannot be deduced whether a point in a real metric
	space $\left(\mathbb{R}^{n},d_{\infty}\right)$ belongs to a subset
	$A$ or to a subset $B$ if $A\cap B$ has a dimension less than $n$.
	However, comparison of a real number with a non-trivial interval is decidable \ie whether a real number in the interval $\left[a,b\right]\subset\mathbb{R}$
	belongs to a non-trivial interval $I_{1}\subseteq\left[a,b\right]$
	or to $I_{2}\subseteq\left[a,b\right]$ if and only if
	$I_{1}\cap I_{2}$ is a non-trivial interval. In the proposition above, this fact is generalized to overlapping hypercubes.
\end{rem}

Further, an important result, called \emph{constructive Arzela--Ascoli's
	lemma}, which is due to \citet[p. 100]{Bishop1985-constr-analysis}, is recalled:

\begin{prop}
	Let $\mathcal{F}$ be an equicontinuous
	function space from a totally bounded metric space $\left(X,\rho\right)$
	to a metric space $\left(Y,\sigma\right)$. Suppose that for any finite\\
	$\frac{1}{k}$--approximation $\left\{ x_{1}, \dots, ,x_{N}\right\} $
	to $X$, the set $A:=\left\{ \left(f\left(x_{1}\right), \dots, f\left(x_{N}\right)\right):f\in\mathcal{F}\right\} \subset Y^{N}$
	is totally bounded. Then, $\mathcal{F}$ is totally bounded.
	\label{lem:Ascoli-Arzel-constr}
\end{prop}

\begin{proof}
	Let $\omega$ be the continuity modulus for the function space $\mathcal{F}$.
	Let $k\in\mathbb{N}$, and let $\left\{ x_{1}, \dots, ,x_{N}\right\} $
	be an $\omega\left(\frac{1}{3k}\right)$--approximation to $X$. By
	assumption, the set $A$ is totally bounded. Let $\left\{ f_{1}, \dots, ,f_{M}\right\} $
	be a set of functions in $\mathcal{F}$ such that the points $a_{i}:=\left(f_{i}\left(x_{1}\right), \dots, f_{i}\left(x_{N}\right)\right),$$i=1, \dots, M$
	form a finite $\frac{1}{4k}$--approximation to $X$. Then, for an
	arbitrary $f\in\mathcal{F}$, it follows that there exists an $f_{i}$
	such that $\sum_{j=1}^{N}\sigma\left(f_{i}\left(x_{j}\right),f\left(x_{j}\right)\right)\le\frac{1}{4k}$.
	For an arbitrary $x\in X$, there exists an $x_{j}$ such that $\rho\left(x,x_{j}\right)\leq\omega\left(\frac{1}{3k}\right)$.
	Then,
	\begin{align*}
		& \sigma\left(f_{i}(x),f(x)\right) \leq \\ 
		& \sigma\left(f_{i}(x),f_{i}\left(x_{j}\right)\right)+\sigma\left(f_{i}\left(x_{j}\right),f\left(x_{j}\right)\right)+\sigma\left(f_{i}\left(x_{j}\right),f\left(x\right)\right) \leq \\
		& \dfrac{1}{3k}+\dfrac{1}{4k}+\dfrac{1}{3k}.
	\end{align*}
	
	It follows that $\tau\left(f_{i},f\right)\leq\frac{1}{k}$. Therefore,
	$\left\{ f_{1}, \dots, ,f_{M}\right\} $ is a finite $\frac{1}{k}-$approximation
	to $\mathcal{F}$ whence $\mathcal{F}$ is totally bounded.
\end{proof}


\section{Main results\label{sec:main-result}}

Based on the preliminaries of the previous section, the new result on approximate optimal control policies can be derived. This is made in two steps. First, it is shown that certain function spaces, which represent the sets of admissible control policies, admit finite approximations provided that they satisfy certain assumptions which are, however, applicable in practice.

\subsection{Finite approximations to function spaces}

The central theorem of this section is stated as follows:

\begin{prop}
	Let $\mathcal{F}$ be the space
	of uniformly Lipschitz and uniformly bounded real--valued functions
	on a totally bounded metric space $\left(X,\rho\right)$. Then $\mathcal{F}$
	is totally bounded.
	\label{thm:constr-tot-bound-fnc-spcs}
\end{prop}

\begin{proof}
	Let $X_{0}$ be a finite subset of $X$ consisting of unequal points
	$\left\{ x_{1}, \dots, ,x_{N}\right\} ,$ $N\in\mathbb{N}$. Suppose that
	$L$ and $K$ are the uniform Lipschitz constant and uniform bound
	for $\mathcal{F}$ respectively. First, show that the subset
	\[
	Y:=\left\{ \left(f\left(x_{1}\right), \dots, f\left(x_{N}\right)\right):f\in\mathcal{F}\right\} 
	\]		
	of $\mathbb{R}^{N}$ with the product metric is totally bounded. To
	this end, let $P=\left(p_{1}, \dots, ,p_{M}\right),M\in\mathbb{N}$ be
	a regular partition of $\left[-K,K\right]$ with a step $\delta:=\frac{1}{k},k\in\mathbb{N}$.
	Let $f$ be any function from the function space $\mathcal{F}$ and
	fix some arbitrary $n\in\mathbb{N}$. Construct a piece--wise linear
	function $\varphi:X\longrightarrow\mathbb{R}$ such that $\forall x,y\in X.|\varphi(x)-\varphi(y)|\leq L\rho(x,y)$
	and $\forall x_{i}\in X_{0}.\big|f\left(x_{i}\right)-\varphi\left(x_{i}\right)\big|\le\frac{1}{nN}$.
	By the product metric, the latter condition would imply that $\big|\big|\left(f\left(x_{1}\right), \dots, f\left(x_{N}\right)\right)-\left(\varphi\left(x_{1}\right), \dots, \varphi\left(x_{N}\right)\right)\big|\big|\leq\frac{1}{n}$.
	First, the image of $\varphi$ on $X_{0}$ is constructed inductively.
	By Proposition \ref{lem:detection-in-hypercubes}, for any $x\in X$ and
	$f\in\mathcal{F}$, there exists a $p_{i}\in P$ such that $\big|f(x)-p_{i}\big|\le\delta$.
	Suppose that $f\left(x_{1}\right),f\left(x_{2}\right)$ are within
	some closed balls $\bar{\mathcal{B}}\left(p_{j_{1}},\delta\right),\bar{\mathcal{B}}\left(p_{j_{2}},\delta\right),j_{1},j_{2}\in\left\{ 1, \dots, M\right\} $
	respectively. Let $\varphi\left(x_{1}\right):=p_{j_{1}}$. Observe
	that since $\big|f\left(x_{1}\right)-f\left(x_{2}\right)\big|\le L\rho\left(x_{1},x_{2}\right)$,
	it follows that $\big|p_{j_{1}}-p_{j_{2}}\big|\le L\rho\left(x_{1},x_{2}\right)+2\delta$.
	Notice that $p_{j_{1}}$and $p_{j_{2}}$ are rational numbers. It can, therefore, be assumed that either $p_{j_{1}}-p_{j_{2}}>2\delta$, or
	$p_{j_{1}}-p_{j_{2}}<-2\delta$, or $\big|p_{j_{1}}-p_{j_{2}}\big|\le2\delta$.
	The first two cases are analogous whence one may assume that $p_{j_{1}}-p_{j_{2}}>2\delta$.
	Let $\varphi\left(x_{2}\right):=p_{j_{2}}-2\delta$. This setting
	ensures the Lipschitz condition. Indeed,
	\begin{align*}
		& \big|\varphi\left(x_{1}\right)-\varphi\left(x_{2}\right)\big|=p_{j_{2}}-2\delta-p_{j_{1}} \le \big|p_{j_{1}}-p_{j_{2}}\big|-2\delta\le L\rho\left(x_{1},x_{2}\right).
	\end{align*}
	
	On the other hand, since $f\left(x_{2}\right)\in\bar{\mathcal{B}}\left(p_{j_{2}},\delta\right)$
	and by the setting of $\varphi\left(x_{2}\right)$ it follows that
	$f\left(x_{2}\right)\in\bar{\mathcal{B}}\left(\varphi\left(x_{2}\right),\delta+2\delta\right)$.
	If $\big|p_{j_{2}}-p_{j_{1}}\big|\leq2\delta$, then setting $\varphi\left(x_{2}\right):=p_{j_{2}}$
	ensures the same conditions. Suppose now that, at the step $i$, $f\left(x_{i}\right)\in\bar{\mathcal{B}}\left(\varphi\left(x_{2}\right),\delta+2(i-1)\delta\right)$.
	Assume that $f\left(x_{i+1}\right)\in\bar{\mathcal{B}}\left(p_{j_{i+1}},\delta\right),j_{i+1}\in\left\{ 1, \dots, M\right\} $.
	Following exactly the same procedure, one may pick the next value of
	$\varphi$ so that the approximation radius grows by $2\delta$ whereas
	the Lipschitz condition is satisfied. After the step $N$, it holds that
	$f\left(x_{N}\right)\in\bar{\mathcal{B}}\left(\varphi\left(x_{N}\right),\delta+2(N-1)\delta\right)$.
	Therefore, setting $k$ equal to $2nN^{2}$ ensures
	$\big|\big|\left(f\left(x_{1}\right), \dots, f\left(x_{N}\right)\right)-\left(\varphi\left(x_{1}\right), \dots, \varphi\left(x_{N}\right)\right)\big|\big|\leq\frac{1}{n}$.
	Now, extend $\varphi$ to the whole space $X$. To this end, let
	\begin{align*}
		& \psi(x):= \frac{1}{2}\left[\max_{i}\left(\varphi\left(x_{i}\right)-L\rho(x,x_{i})\right)+\min_{i}\left(\varphi\left(x_{i}\right)+L\rho(x,x_{i})\right)\right], \forall x\in X.
	\end{align*}
	
	Proposition \ref{lem:max-2-reals} implies that $\psi(x_{j})=\varphi(x_{j})$.
	It follows from the fact that $\forall j,\varphi(x_{i})-L\rho(x_{j},x_{i})\le\varphi(x_{j})$.
	To see that $\psi(x)$ is an $L$--Lipschitz function, observe that
	$\rho(x,x_{i})$ is a $1$--Lipschitz function of $x$ whence $\varphi\left(x_{i}\right)-L\rho(x,x_{i})$
	is an $L$--Lipschitz function of $x$ for each $i$. Therefore, $\max_{i}\left(\varphi\left(x_{i}\right)-L\rho(x,x_{i})\right)$
	and $\min_{i}\left(\varphi\left(x_{i}\right)+L\rho(x,x_{i})\right)$
	are uniformly continuous functions with the same Lipschitz constant
	$L$. The factor $\frac{1}{2}$ ensures that $\psi(x)$ is $L$--Lipschitz.
	Let $\varphi(x):=\max\{\min\{\psi(x),K\},-K\}$. Due to the properties
	of the minimum and maximum \citep[p.~23]{Bishop1985-constr-analysis},
	it follows that $\varphi$ is an $L$--Lipschitz continuous function
	satisfying $\varphi(x)\le K$ and $\varphi(x)\ge-K$. Thus, $\varphi$
	belongs to the function space $\mathcal{F}$ and approximates $f$
	at the points $X_{0}$ arbitrarily close. Since $\varphi$ is uniquely
	defined by its values at $X_{0}$, and the values $\left\{ \varphi\left(x_{i}\right):x_{i}\in X_{0}\right\} $
	take place in a finite set $P$, whereas the distances between the
	function values at each two points of $X_{0}$ have fixed bounds,
	there are finitely many such functions. Further, since $f$ was arbitrary,
	it follows that $Y$ is totally bounded. By the constructive Arzela--Ascoli's
	proposition \ref{lem:Ascoli-Arzel-constr}, the function space $\mathcal{F}$
	is totally bounded.
\end{proof}

\begin{cor}
	Let $\mathcal{F}$ be the space of uniformly Lipschitz and uniformly
	bounded functions from a totally bounded metric space $\left(X,\rho\right)$
	to $\mathbb{R}^{m},m\in\mathbb{N}$ with the $d_{\infty}$--metric.
	Then $\mathcal{F}$ is totally bounded.
\end{cor}

\begin{proof}
	The proof amounts to the same procedure, as in the proof of the theorem,
	done for each coordinate separately since 
	\begin{align*}
		& \forall i=1, \dots, m,\big|x^{i}-y^{i}\big| \le \varepsilon \iff \|x-y\|_{\infty}=\max_{i}\big|x^{i}-y^{i}\big| \le \varepsilon
	\end{align*}
	for any $x$ and $y$ in $\left(\mathbb{R}^{m},d_{\infty}\right)$
	and $\varepsilon>0$. 
\end{proof}

\begin{rem}
	The same result applies if there is a uniform bound and uniform Lipschitz
	constant for each dimension separately: $\exists\left(K_{1}, \dots, K_{m}\right),\forall f\in\mathcal{F},\forall x\in X,$$|f^{i}(x)|\leq K_{i},i=1, \dots, m$
	and $\exists\left(L_{1}, \dots, L_{m}\right),$ $\forall f\in\mathcal{F},\forall x,y\in X,$$\big|f^{i}(x)-f^{i}(y)\big|\leq L\rho(x,y),i=1, \dots, m$.
	The proof is by rescaling of the hypercuboid with the side lengths
	$\left(2K_{1}, \dots, 2K_{m}\right)$ centered at the origin.
\end{rem}

\begin{cor}
	Let $\mathcal{\bar{C}}^{1}$ be the space of uniformly bounded functions from a compact set $X\subset\mathbb{R}^{n},n\in\mathbb{N}$
	to $\mathbb{R}^{m},m\in\mathbb{N}$ with the $d_{\infty}$--metric,
	and suppose that the derivatives of the functions in $\mathcal{\bar{C}}^{1}$
	are uniformly bounded. Then $\mathcal{\bar{C}}^{1}$ is totally bounded.
	\label{cor:C-1-bar-tot-bounded}
\end{cor}

\begin{proof}
	Let $\mathcal{F}$ denote the space of functions as described in
	the theorem. Fix an arbitrary function $g$ from $\mathcal{\bar{C}}^{1}$.
	Clearly, $g$ is a function in $\mathcal{F}$ whence $\mathcal{\bar{C}}^{1}\subset\mathcal{F}$.
	The converse is not true. It suffices to show that $\mathcal{\bar{C}}^{1}$ is dense in $\mathcal{F}$. Fix an arbitrary function $f\in\mathcal{F}$
	and a number $k\in\mathbb{N}$. Following the construction as in the
	proof, one can derive a piece--wise linear function $\varphi:\mathbb{R}^{n}\rightarrow\left(\mathbb{R}^{m},d_{\infty}\right)$
	that approximates $f$ on $X$ up to the precision $\frac{1}{2k}$.
	Further, one can construct an analytic function $\varphi_{2k}$ that
	approximates $\varphi$ up to the precision $\frac{1}{2k}$ and has
	the same Lipschitz constant (see details in Appendix).
	Therefore, $\varphi_{k}$ approximates $f$ up to the precision $\frac{1}{k}$.
	Therefore, the set $\mathcal{\bar{C}}^{1}$, as a dense subset of
	a totally bounded set $\mathcal{F}$, is itself totally bounded \citep[p. 28]{Bridges1987-varieties}.
\end{proof}

\begin{cor}
	Let $\mathcal{\bar{C}}^{N},N\in\mathbb{N}$
	be the space of uniformly bounded functions from a compact set $X\subset\mathbb{R}^{n},n\in\mathbb{N}$
	to $\mathbb{R}^{m},m\in\mathbb{N}$ with the $d_{\infty}$--metric,
	and suppose that the derivatives of the functions in $\mathcal{\bar{C}}^{N}$
	up to order $N$ are uniformly bounded. Then $\mathcal{\bar{C}}^{N}$
	is totally bounded.
	\label{cor:smooth-fncs-tot-bounded}
\end{cor}

\subsection{Approximate extrema}

In this section, based on the construction of approximating functions of the previous section, the new constructive version of the approximate extremum value theorem for function spaces is stated. The implication of it is that, under certain assumptions, approximate control policies may be computed up to a prescribed accuracy.
\begin{thm}
	Let $\mathcal{F}$ be the space of uniformly Lipschitz
	and uniformly bounded functions from a totally bounded metric space
	$\left(X,\rho\right)$ to $\mathbb{R}$, and let $J$ be a uniformly
	continuous functional from $\mathcal{F}$ to $\mathbb{R}$. Then,
	for any $k\in\mathbb{N}$, there exists an $f\in\mathcal{F}$ such
	that $J[f]-\frac{1}{k}\leq\inf J$.
	\label{thm:aEVT-fnc}
\end{thm}

\begin{proof}
	Since $\mathcal{F}$ is totally bounded by Theorem \ref{thm:constr-tot-bound-fnc-spcs}
	and $J$ is uniformly continuous, $\inf J$ exists. Let $\alpha$
	be the continuity modulus of $J$ and $\mathcal{F}_{0}=\left\{ f_{1}, \dots, ,f_{N}\right\} $
	be an $\alpha\left(\frac{1}{8k}\right)-$approximation to $\mathcal{F}$.
	Consider all finitely many $\left\{ J\left[f_{i}\right](8k)\right\} ,i=1, \dots, N$.
	Let $J\left[f_{j}\right](8k)$ be the smallest one, and such that
	$j$ is the smallest index if there are more than one such indices.
	Observe that $\big|J\left[f_{j}\right]-J\left[f_{j}\right](8k)\big|\leq\frac{1}{4k}$
	and $\forall f\in\mathcal{F},\big|\big|f_{j}-f\big|\big|\leq\alpha\left(\frac{1}{8k}\right)$$\implies\big|J(f)-J\left[f_{j}\right]\big|\leq\frac{1}{4k}$
	whence $J\left[f_{j}\right](8k)-\frac{1}{2k}\leq J[f](8k)$. Therefore,
	$J\left[f_{j}\right](8k)-\frac{1}{2k}\leq J[f]$ and consequently
	$J\left[f_{j}\right]-\frac{1}{k}\leq J[f]$. The same holds trivially
	if $\big|\big|f_{j}-f\big|\big|>\alpha\left(\frac{1}{8k}\right)$.
	Since $f$ is arbitrary, $J\left[f_{j}\right]-\frac{1}{k}$ is a lower
	bound of $J$ and so, in particular, $\inf J\geq J\left[f_{j}\right]-\frac{1}{k}$.
\end{proof}

\begin{cor}
	Let $\mathcal{F}$ be the space of uniformly Lipschitz and uniformly
	bounded functions from a totally bounded metric space $\left(X,\rho\right)$
	to $\mathbb{R}$, and let $J$ be a uniformly continuous functional
	from $\mathcal{F}$ to $\mathbb{R}^{m},m\in\mathbb{N}$ with the $d_{\infty}$--metric.
	Then, for any $k\in\mathbb{N}$, there exists an $f\in\mathcal{F}$
	such that $J[f]-\frac{1}{k}\leq\inf J$.
\end{cor}

\begin{rem}
	To compute an extremal function which yields a $\frac{1}{k}$-infimum of $J$, where $k$ describes the specified precision, construct all possible piece-wise linear functions over the regular partition of step $\frac{1}{2 k N^2}$ where $N$ is as in Theorem \ref{thm:aEVT-fnc} by preserving the common Lipschitz constant. Smoothen the constructed functions, if necessary, as per Corollary \ref{cor:C-1-bar-tot-bounded}. Then, choose a one $f_j$ which satisfies $J[f]-\frac{1}{k}\leq\inf J$. 
\end{rem}

\begin{rem}
	Theorem \ref{thm:aEVT-fnc} describes the worst-case scenario a numerical algorithm can perform in general. Various numerical approaches exist and they may be numerically fast, but the best one can expect in general is the result as in the statement of the theorem. The implication for optimal control is that optimality may fail to be achieved in general. Instead, approximate optimal control policies can be effectively computed, provided that the system and the cost function satisfy the assumptions in the statement of the theorem. These assumptions are, however, practicable as justified by physical nature of the control problems and demonstrated in the next section on finite-horizon optimal control, dynamic programming and adaptive dynamic programming.
	\label{rem:aEVT-fnc-implication}
\end{rem}

\begin{rem}
	The statement for the supremum is equivalent.
\end{rem}


\section{Case study: optimal control\label{sec:case-study}}

In this section, the derived version of a constructive extremum value theorem in application to finite--horizon optimal control,
dynamic programming (DP), and adaptive dynamic programming (ADP), is discussed. 

\subsection{Finite-horizon optimal control}

Classical theorems of existence of extremal solutions to functional
optimization problems essentially rely on Bolzano-Weierstrass's theorem
that every bounded sequence has a convergent subsequence. One first
shows that the function space in question is compact, and applies
the sequential compactness argument. There is, unfortunately, no constructive way to find a convergent subsequence. Therefore, approximate solutions are investigated in this section. Recall the problem of minimization of the following cost functional:
\begin{equation}
	J[u]:=\varphi\left(x\left(t_{1}\right)\right)+\int\limits _{t_{0}}^{t_{1}}\mathcal{L}(x(t),u(x(t)),t)dt
	\label{eq:finite-horizon-cost}
\end{equation}
subject to $\dot{x}(t):=f(x(t),u(t),t),x\left(t_{0}\right)=x_{0}$. Here, $\mathcal L$ is the running cost, or Lagrangian, which is usually a positive-definite function of $x, u, t$.
Assume that the state space $X\subset\left(\mathbb{R}^{n},d_{\infty}\right),n\in\mathbb{N}$
is compact. With the $d_{\infty}$--metric, two states $x$ and $y$
are close whenever their respective components $x^{i},i=1, \dots, n$ and
$y^{i},i=1, \dots, n$ are close. Therefore, a state trajectory $x(t)$
is uniformly continuous whenever each state component $x^{i}(t)$
is uniformly continuous. It can be assumed that $u(x)\in\left(\mathbb{R}^{m},d_{\infty}\right),m\in\mathbb{N}$
for any $x\in X$. Let $\mathcal{U}$ denote the set of admissible control policies \ie those which yield state trajectories within $X$. In \eqref{eq:finite-horizon-cost}, the
starting time $t_{0}\in\mathbb{Q}$ and the final time $t_{1}\in\mathbb{Q}$
are assumed fixed. Suppose that $f:X\times U\times\mathbb{R}\rightarrow X$
satisfies the Lipschitz condition for $x$ and $u$ on $X\times U$
in the following sense:
\begin{equation}
	\|f(x,u,t)-f(y,v,t)\|\le L_{f}\max\left\{ \|x-y\|,\|u-v\|\right\} \label{eq:Lipschitz-cond-IVP}
\end{equation}
for some rational $L_{f}>0$. Then, the constructive theorem of existence
and uniqueness of solutions of the initial value problem $\dot{x}(t)=f(x(t),u(t),t),x\left(t_{0}\right)=x_{0},t\in\left[t_{0},t_{1}\right]$
applies \citep[12.4]{Schwichtenberg2012-constr-analys} provided that $u(t)$ is continuous.
Further, $\varphi$ is assumed to be uniformly continuous on $X\times X$.
The Lagrangian should be also uniformly continuous on $X\times U$:
\begin{align*}
	& \exists\omega_{\mathcal{L}}:\mathbb{Q}\rightarrow\mathbb{Q},\forall k\in\mathbb{N},\forall x,y\in X,\forall u,v\in U,\forall t\in\left[t_{0},t_{1}\right], \\
	& \max\left\{ \|x-y\|,\|u-v\|\right\} \leq\omega_{\mathcal{L}}\left(\frac{1}{k}\right) \implies \\
	& \big|\mathcal{L}(x,u,t)-\mathcal{L}(y,v,t)\big|\leq\frac{1}{k}.
\end{align*}

Consider two control policies $u(x),v(x),x\in X$ in $\mathcal{U}$ and
the respective state trajectories:
\begin{align*}
	x_{u}(\tau)=\int\limits _{t_{0}}^{\tau}f(x(t),u(x(t)),t)dt,\\
	x_{v}(\tau)=\int\limits _{t_{0}}^{\tau}f(x(t),v(x(t)),t)dt
\end{align*}
for an arbitrary $\tau\in\left[t_{0},t_{1}\right]$. It follows that
\begin{align*}
	& \bigg|\bigg|\int\limits _{t_{0}}^{\tau}\left(f(x(t),u(x(t)),t)-f(x(t),v(x(t)),t)\right)dt\bigg|\bigg|\leq \\
	& \sup_{t_{0}\leq\tau\le t_{1}}\|f(x(t),u(x(t)),t)-f(x(t),v(x(t)),t)\|\cdot\left(\tau-t_{0}\right)\leq \\
	& L_{f}\cdot\underset{x\in X}{\sup}\|u(x)-v(x)\|\cdot\left(\tau-t_{0}\right)\leq \\
	& L_{f}\cdot\|u-v\|\cdot\left(t_{1}-t_{0}\right).
\end{align*}

Therefore, if $\underset{x\in X}{\sup}\|u(x)-v(x)\|\leq\frac{1}{k},k\in\mathbb{N}$,
which is to say that $\|u-v\|\leq\frac{1}{k}$, then $\big|\big|x_{u}-x_{v}\big|\big|\leq L_{f}\frac{t_{1}-t_{0}}{k}$.
Consequently, for $k\ge L_{f}\left(t_{1}-t_{0}\right)$, and $\|u-v\|\leq\omega_{\mathcal{L}}\left(\frac{1}{k}\right)$
it follows that: $\big|\mathcal{L}\left(x_{u}(t),u(t),t\right)-\mathcal{L}\left(x_{v}(t),v(t),t\right)\big|\leq\frac{1}{k}$.
If $k<L_{f}\left(t_{1}-t_{0}\right)$, then $\big|\mathcal{L}\left(x_{u}(t),u(t),t\right)-\mathcal{L}\left(x_{v}(t),v(t),t\right)\big|\leq L_{f}\frac{t_{1}-t_{0}}{k}$
whence the continuity modulus is easily derived. Further, it holds that:
\begin{align*}
	& \bigg|\int\limits _{t_{0}}^{t_{1}}\left(\mathcal{L}(x(t),u(x(t)),t)-\mathcal{L}(x(t),v(x(t)),t)\right)dt\bigg|\leq \frac{t_{1}-t_{0}}{k},
\end{align*}

In case if $\mathcal{U}$ is a located subset of the space of uniformly bounded and uniformly Lipschitz functions on $X$ with a uniform bound and uniform Lipschitz constant, respectively (which can physically be dictated by the fact that any control policy has limits on magnitude and rate
of change with respect to the state), by Theorem \ref{thm:constr-tot-bound-fnc-spcs} and Lemma 4.3. in \citep{Ye2011-SF}, $\mathcal{U}$ is totally bounded. Since $J$ is a uniformly continuous functional from $\mathcal{U}$ to $\mathbb{R}$, by Theorem \ref{thm:aEVT-fnc},
for any $k\in\mathbb{N}$, there exists a control policy $u^{*}(x),x\in X$
such that $J\left[u^{*}\right]-\frac{1}{k}\leq J[u]$ for any other
control policy $u(x),x\in X$.

This implies that control policies, which yield approximate optima of the cost functional, can be effectively computed provided that the controllers have bounds on the magnitude and rate of change of the controls which is satisfied in practice due to physical nature. The next section considers infinite-horizon optimal control in the framework of dynamic programming.

\subsection{Dynamic programming\label{sub:DP}}

In dynamic programming, optimization problems in the following form
are considered:
\begin{equation}
	\sup_{u\in U}\left\{ r\left(x,u\right)+\gamma V(f(x,u))\right\} ,\forall x\in X,\label{eq:DP-opt-problem}
\end{equation}

In \eqref{eq:DP-opt-problem}, $u$ is taken from a totally bounded set $U\subset\mathbb{R}^{m},m\in\mathbb{N}$.
In the case when each component $u^{i},i=1, \dots, m$ is an independent
function, the $d_{\infty}$--metric may be assumed on $\mathbb{R}^{m}$.
The function $r:X\times U\longrightarrow\mathbb{R}$
is a positive--definite \emph{utility function (or running cost)} that describes the instantaneous cost whereas $V$ is the \emph{value function} that
describes the cumulative cost. The functions $r$ and $V$ are assumed
to be bounded on their domains. The parameter $0<\gamma<1$ is called
\emph{discounting factor}. Finally, the \emph{dynamic programming
	operator} is introduced:
\begin{equation}
	T[V](x):=\sup_{u\in U}\left\{ r\left(x,u\right)+\gamma V(f(x,u))\right\} ,\forall x\in X.\label{eq:DP-opt-problem-operator}
\end{equation}

The operator $T$ acts on the space of continuous and bounded functions.
The \emph{Hamilton--Jacobi--Bellman} equation is defined by the fixed--point
of $T$. The natural question is whether $T$ yields continuous functions,
whether the extrema of $r\left(x,u\right)+\gamma V(f(x,u))$ over
$U$ exist, and whether they are continuous in $x$ in a certain sense.
The answer is given by the Berge's Theorem of the Maximum \citep[p. 115]{Berge1963}.
It shows what kind of continuity of the extrema is preserved if the
optimization problem is continuous in a certain sense. First, recall
the  definition of hemi--continuity that generalizes the notion of
continuity to \emph{multi--functions} \ie functions from a set
to the power set of another set. A multi--function is called \emph{compact--valued} if its values are compact sets.

\begin{defn}
	(Upper hemi--continuity) Let $\Gamma:X\rightarrow U$ be a multi--function
	such that $\Gamma(x)$ is closed for all $x$ in $X$. Then, $\Gamma$
	is called upper hemi--continuous at $x\in X$ if for any sequence
	$\left\{ x_{n}\right\} _{n}$ in $X$, $u$ in $U$ and sequence $\left\{ u_{n}\right\} _{n}$
	such that $u_{n}\in\Gamma\left(x_{n}\right)$, it follows that
	\[
	\left(\lim_{n\rightarrow\infty}x_{n}=x\land\lim_{n\rightarrow\infty}u_{n}=u\right)\implies u\in\Gamma(x).
	\]
	
\end{defn}

\begin{defn}
	(Lower hemi--continuity) A multifunction $\Gamma:X\rightarrow U$
	is called lower hemi--continuous at $x\in X$ if for any sequence
	$\left\{ x_{n}\right\} _{n}$ in $X$ such that $\lim_{n\rightarrow\infty}x_{n}=x$,
	any $u\in\Gamma(x)$, there exists a subsequence $\left\{ x_{nk}\right\} _{k}\subset\left\{ x_{n}\right\} _{n}$
	such that there exist $u_{k}\in\Gamma\left(x_{nk}\right)$ with $\lim_{k\rightarrow\infty}u_{k}=u$.
\end{defn}

\begin{thm}
	(Theorem of the Maximum) Let $f$ be a jointly
	continuous function from a product of two metric spaces $\left(X,\rho\right)\times\left(U,\sigma\right)$
	to $\mathbb{R}$, and $\Gamma:X\rightarrow U$ be a compact--valued
	upper and lower hemi--continuous multi--function. Then, the function
	$h(x):=\underset{u\in\Gamma(x)}{\sup}f(x,u)$ is continuous, and $u^{*}(x):=\arg\underset{u\in\Gamma(x)}{\sup}f(x,u)$
	is a non--empty, compact--valued, upper hemi--continuous multi--function.
	\label{thm:max-thm}
\end{thm}

In the setting of \eqref{eq:DP-opt-problem}, the multi--function
$\Gamma$ is taken as a constant multi--function $\Gamma(x)\equiv U$
with the assumption that any control action is available at any state.
The proof of Theorem \ref{thm:max-thm} essentially uses the classical
extremum value theorem that is not valid constructively. A constructive
analysis of the Maximum Theorem has been done by \citet{Tanaka2012}.
To prove a constructive version of Theorem \ref{thm:max-thm}, Tanaka
introduces the notion of a function with \emph{sequentially locally
	at most one maximum}. Such a condition is, however, hard to verify
in practice. To summarize, approximate extrema of \eqref{eq:DP-opt-problem} cannot be shown uniformly continuous in general, let alone found exactly. However, a weaker result holds:

\begin{prop}
	Let $f$ be a uniformly continuous function
	from a product of two totally bounded metric spaces $\left(X,\rho\right)\times\left(U,\sigma\right)$
	to $\mathbb{R}$. Then, the function $h(x):=\underset{u\in U}{\sup}f(x,u)$
	is uniformly continuous.
	\label{prop:appr-max-thm}
\end{prop}

\begin{proof}
	Fix any $x,y\in X$ and $k\in\mathbb{N}$. By Theorem \ref{thm:aEVT-fnc},
	there exists $u^{*}\in U$ such that $f\left(x,u^{*}\right)+\frac{1}{k}\ge\underset{u\in U}{\sup}f(x,u)$.
	Let $\omega$ be the continuity modulus of $f$. Fix any $y\in X$
	such that $\rho(x,y)\leq\omega\left(\frac{1}{k}\right)$. It follows
	that $\forall u\in U,|f(x,u)-f(y,u)|\leq\frac{1}{k}$. In particular,
	$|f\left(x,u^{*}\right)-f\left(y,u^{*}\right)|\leq\frac{1}{k}$ whence
	$f\left(y,u^{*}\right)+\frac{1}{k}\ge f\left(x,u^{*}\right)$. Therefore,
	$f\left(y,u^{*}\right)+\frac{2}{k}\ge\underset{u\in U}{\sup}f(x,u)$.
	From the continuity condition, it follows that $f(y,u)\leq f(x,u)+\frac{1}{k}$
	whence $f\left(y,u^{*}\right)+\frac{3}{k}\ge f(y,u)$ for all $u\in U$.
	It means that $f\left(y,u^{*}\right)+\frac{3}{k}\ge\underset{u\in U}{\sup}f(y,u)$.
	Finally, $\Big|\underset{u\in U}{\sup}f(x,u)-\underset{u\in U}{\sup}f(y,u)\Big|\leq|f\left(x,u^{*}\right)-f\left(y,u^{*}\right)|+\frac{2}{k}=\frac{3}{k}$.
	It follows that $h$ is a uniformly continuous function with a modulus
	$k\mapsto\omega\left(\frac{3}{k}\right)$.
\end{proof}

It can be proven that $T$ is indeed an operator that sends uniformly
continuous functions to uniformly continuous functions provided that
$r,V,f$ are uniformly continuous functions. In turn, uniformly continuous
functions on totally bounded sets are bounded. Denote the set of uniformly
continuous functions from $X$ to $X$ by $\mathcal{V}$. Assuming
the supremum norm on $\mathcal{V}$, it is easy to show that $T$
is a contraction mapping. The proof is already constructive. First,
observe that $r\left(x,u\right)+\gamma V(f(x,u))\leq r\left(x,u\right)+\gamma W(f(x,u))$
whenever $V(x)\leq W(x)$ for any $x$ in $X$. Therefore,
\begin{align*}
	& \sup_{u\in U}\left\{ r\left(x,u\right)+\gamma V(f(x,u)))\right\} \leq \sup_{u\in U}\left\{ r\left(x,u\right)+\gamma W(f(x,u)))\right\} .
\end{align*}

Consequently, $T[V]\leq T[W]$. It constitutes the monotonicity of
$T$ which is the first Blackwell's sufficient condition for a contraction
mapping \citep{Blackwell1965-DP}. The second condition requires that
$T$ be discounting:
\begin{align*}
	T[V+a](x)=\sup_{u\in U}\left\{ r\left(x,u\right)+\gamma V(f(x,u))+\gamma a\right\} = T[V](x)+\gamma a
\end{align*}
for any $a\ge0$. It follows that $T$ is a contraction mapping with
a modulus $\gamma$. By the Banach fixed point theorem \citep{Palais2007-Banach-fixed-pt-thm},
$T$ has a unique fixed point. The proof of the theorem is essentially
constructive and provided by the algorithm:
\[
V_{n}:=T\left[V_{n-1}\right],n\in\mathbb{N}
\]
starting from an arbitrary uniformly continuous function $V_{0}$
from which it immediately follows that $\left\{ V_{n}\right\} _{n}$
is a Cauchy sequence with the modulus defined from:
\[
\big|\big|T^{n}\left[V_{0}\right]-V^{*}\big|\big|\leq\dfrac{\gamma^{n}}{1-\gamma}\big|\big|T\left[V_{0}\right]-V_{0}\big|\big|,
\]
where $T^{n}[V]=\underbrace{T[T[ \dots T[V]]]}_{n\text{ times}}$ and
$V^{*}$ is the fixed point. Thus, $V_{n}$ converges to $V^{*}$
uniformly and $V^{*}$ is in turn a uniformly continuous function.
The only difference to the classical theorem is that $\mathcal{V}$
must be non-empty: $\exists V_{0}\in\mathcal{V}.$ Using this fact,
one can perform \emph{value iteration} starting from any $V_{0}\in\mathcal{V}$
and stopping at some $V_{n}$ such that a convergence criterion $\big|\big|V_{n}-V^{*}\big|\big|\le\varepsilon$
is satisfied. Notice that the number of steps $n$ can be directly
determined from the desired accuracy $\varepsilon$. Having found
a suitable approximate $V_{n}$, an optimal control
policy is of interest. In the constructive setting, a control policy must be a uniformly
continuous function. In real--world applications, each control action
has physical limits for magnitude and rate of change. Thus, one may
assume that the space of admissible control policies $\mathcal{U}$ is a located subset of the space of uniformly
bounded and uniformly Lipschitz functions on $X$ with a common bound and Lipschitz constant, respectively. 
By Theorem \ref{thm:constr-tot-bound-fnc-spcs},
the latter space is totally bounded, and since $\mathcal{U}$ is located, it is also totally bounded \citep[Lemma~4.3.]{Ye2011-SF}. To cope with the problem of continuity
of extrema in the state variable, it is suggested to consider the following
relaxed optimization problem for $V_{n}$:
\begin{equation}
	\sup_{u\in\mathcal{U}}\inf_{x\in X}\left\{ r\left(x,u(x)\right)+\gamma V_{n}(f(x,u(x)))\right\} .\label{eq:DP-opt-problem-altern}
\end{equation}

It follows that $J:\mathcal{U}\rightarrow\mathbb{R}$ defined by 
\[
J[u]:=\inf_{x\in X}\left\{ r\left(x,u(x)\right)+\gamma V_{n}(f(x,u(x)))\right\} 
\]
is a uniformly continuous functional since $r,V_{n},f,\inf$ are uniformly
continuous. By Theorem \ref{thm:aEVT-fnc}, for any $k\in\mathbb{N}$,
there exists a control policy $u^{*}\in\mathcal{U}$ such that
\begin{align*}
	& \inf_{x\in X}\left\{ r\left(x,u^{*}(x)\right)+\gamma V_{n}(f(x,u^{*}(x)))\right\} +\frac{1}{k}\ge \inf_{x\in X}\left\{ r\left(x,u(x)\right)+\gamma V_{n}(f(x,u(x)))\right\} 
\end{align*}

for any control policy $u$. The difference between \eqref{eq:DP-opt-problem-altern}
and \eqref{eq:DP-opt-problem} lies in the way the performance mark
is defined. In the latter case, $r\left(x,u\right)+\gamma V(f(x,u))$
is optimized in all states $x$ while in the former, the ``worst''
state is optimized. The optimization problem \eqref{eq:DP-opt-problem-altern}
is thus more mild than the original one, but it is still appropriate
for a variety of practical applications.

\subsection{Adaptive dynamic programming}

ADP is a variant of dynamic programming that is suitable for real--time
optimization problems. It may be considered as a reinforcement learning
technique \citep{Sutton1998-RL} in the sense that it uses an iterative
procedure of updating a so--called actor, that produces a control
policy, according to a citric that represents the value function.
The value function in the framework of ADP is commonly a subject to
approximation since exact optimal solutions may be not achieved \citep{Lewis1995}.
In this regard, neural networks are widely used as approximators \citep{Werbos1990-menu,Werbos1992-ADP,Bertsekas1995-NDP}.
For recent surveys on ADP, refer \eg to \citet{Balakrishnan2008-ADP-survey}
and \citet{Ferrari2011-ADP-survey}. It is common to consider ADP
in application to discrete--time systems of the form $x_{k+1}:=f\left(x_{k},u_{k}\right),k\in\mathbb{N}$
or even affine in control: $x_{k+1}:=f\left(x_{k}\right)+g\left(x_{k}\right)u_{k},k\in\mathbb{N}$.
It is also assumed that $f(0)=g(0)=0$, and that there exists a control
policy $u$ such that for all initial conditions $x_{0}\in X$, $x_{k}\rightarrow0$
as $k\rightarrow\infty$. ADP usually addresses the following infinite-horizon
optimization problem:
\begin{equation}
	\inf_{u}\sum_{l=k}^{\infty}r\left(x_{l},u\left(x_{l}\right)\right),\forall x_{k}\in X.\label{eq:cost-to-go}
\end{equation}

\citet{Al-Tamimi2008-VI-ADP} have provided a convergence analysis
of an ADP algorithm for affine--in--control systems. Assuming a utility
function in the form $r\left(x,u\right):=q\left(x\right)+u^{T}Ru$
with $q$ being a positive-definite function, and given an arbitrary
$\mathcal{C}^{\infty}\left(X\subset\mathbb{R}^{n},\mathbb{R}\right)-$function
$V_{0}\left(x\right)$ such that $\forall x\in X,0\leq V\left(x\right)\leq Q\left(x\right),$
perform the following iterations starting with $i:=0$ for all $x_{k}\in X$:

\begin{eqnarray}
	u_{i}\left(x{}_{k}\right): & = & \arg\inf_{u}\left\{ r\left(x_{k},u\right)+V_{i}\left(f\left(x_{k},u\right)\right)\right\} , \label{eq:ADP-VI-ctrl-policy}\\
	V_{i+1}\left(x_{k}\right): & = & r\left(x_{k},u_{i}\left(x_{k}\right)\right)+V_{i}\left(f\left(x_{k},u_{i}\left(x\right)\right)\right),\label{eq:ADP-VI-value-fnc}\\
	i: & = & i+1.\nonumber 
\end{eqnarray}

Al-Tamimi showed that this algorithm converges to the solution of
the Hamilton--Jacobi--Bellman equation

\begin{eqnarray}
	V^{*}\left(x_{k}\right) & = & \inf_{u}\left\{ r\left(x_{k},u\right)+V^{*}\left(f\left(x_{k},u\right)\right)\right\} ,\label{eq:HJB-eqn}\\
	u^{*}\left(x_{k}\right) & = & \arg\inf_{u}\left\{ r\left(x_{k},u\right)+V^{*}\left(f\left(x_{k},u\right)\right)\right\},
	\label{eq:opt-ctrl-policy} \\
	& & \forall x_{k}\in X
\end{eqnarray}
as $i\rightarrow\infty$. The proof essentially uses the classical
monotone convergence theorem that states that a sequence of real numbers
converges whenever it is bounded and monotone. Consequently, no estimate
on the number of iterations can be given for the prescribed accuracy
$\big|\big|V_{i}-V^{*}\big|\big|$. Another subtle point, that is
hard to justify from both the constructive and practical viewpoint,
is the assumption that \eqref{eq:ADP-VI-ctrl-policy} can be solved
in terms of a closed-form expression. That is generally impossible.
An exception is \eg a\emph{ linear quadratic regulator }(LQR)
that is a solution for linear systems. \citet{Liu2014-PI-ADP} introduced
a similar proof technique, as in \citep{Al-Tamimi2008-VI-ADP}, for
a \emph{policy iteration} algorithm: start with $i:=0$, and any continuous
control policy $u_{0}$ such that $u(0)=0$, the state trajectory
$x_{k}\rightarrow0$ under $u_{0}$, and \eqref{eq:cost-to-go} converges,
perform the following iterations:
\begin{eqnarray}
	V_{i}\left(x\right): & = & r\left(x,u_{i}\left(x\right)\right)+V_{i}\left(f\left(x,u_{i}\left(x\right)\right)\right),\label{eq:ADP-PI-value-fnc}\\
	u_{i}\left(x{}_{k}\right): & = & \arg\inf_{u}\left\{ r\left(x_{k},u\right)+V_{i-1}\left(f\left(x_{k},u\right)\right)\right\} ,\label{eq:ADP-PI-ctrl-policy}\\
	& & \forall x_{k}\in X \\
	i: & = & i+1.\nonumber 
\end{eqnarray}

Notice the difference in iteration indices for the value function
and the control policy. Again, the proof of convergence uses the monotone
convergence theorem and the assumption that \eqref{eq:ADP-PI-ctrl-policy}
can be solved in terms of a closed-form expression. To coup with this
problem, Al-Tamimi and Liu suggest to use neural--network based approximators
for the value function and the control policy. Unfortunately, no convergence
proof has been given for such an approximate setting \citep[p. 632]{Liu2014-PI-ADP}.
An alternative approach has been proposed by \citet{Heydari2014-better-prf-ADP}.
Instead of approximating the control policy, Heydari has shown that
the first--order necessary condition for an extremum
\[
u_{i}(x)=-\frac{1}{2}R^{-1}g^{T}(x)\frac{\partial V_{i}\left(f(x)+g(x)u_{i}(x)\right)}{\partial x},\forall x\in X
\]
is a fixed--point equation provided that all the functions in question
are $\mathcal{C}^{\infty}$. By an appropriate choice of the matrix
norm of $R^{-1}$ and/or $g(x)$, it can be shown that the mapping
$F:\mathcal{C}^{\infty}\left(X,\mathbb{R}^{m}\right)\rightarrow\mathcal{C}^{\infty}\left(X,\mathbb{R}^{m}\right)$
defined by
\begin{equation}
	F[u](x):=-\dfrac{1}{2}R^{-1}g^{T}\left(x\right)\frac{\partial V_{i}(f(x)+g(x)u(x))}{\partial x},\label{eq:ctrl-policy-contraction}
\end{equation}

is a contraction. The assumption that the control policy at each iteration
is a smooth function satisfies our argumentation in Section \ref{sub:DP}
and, provided with a uniform bound and Lipschitz constant, leads to
total boundedness of the space of control policies by Corollary \ref{cor:smooth-fncs-tot-bounded}.
However, the first--order condition for an extremum is not sufficient
to claim that the infimum of \eqref{eq:ADP-VI-value-fnc} is attained
at each iteration. Currently, one can decouple \eqref{eq:ADP-VI-ctrl-policy}
and \eqref{eq:ADP-VI-value-fnc}, iterate the value function and then
claim existence of an approximate optimal control policy for an alternative
performance mark \eqref{eq:DP-opt-problem-altern}. Relaxing the continuity
condition by considering measurable functions, and investigating other
performance marks, such as Lebesgue integrals, may be of interest
for future research.


\markboth{P. OSINENKO, S. STREIF}{\rm APPROXIMATE EVT FOR FUNCTIONS SPACES}

\section{Conclusions}

The present work is highlighted in the following points: 
\begin{itemize}
	\item A new constructive proof of the approximate extremum value theorem for function spaces is suggested.
	\item The methodological approach of the proof takes into account the numerical uncertainty which is related to limitations of real number representations in a computational device.
	\item The functions forming finite approximations to the respective function spaces are constructed explicitly. In particular, it was shown that the sets of uniformly bounded and uniformly Lipschitz functions on a totally bounded set are totally bounded by explicit constructions of approximating functions.
	\item As stated in Remark \ref{rem:aEVT-fnc-implication}, any numerical procedure may in general achieve the result of Theorem \ref{thm:aEVT-fnc} at best. That implies that optimality in optimal control problems may in general be achieved only approximately.
	\item Applications of the theorem to finite-horizon optimal control, dynamic programming and adaptive dynamic programming are addressed. It is shown that under the stated assumptions, whose practicability is discussed, approximate optimal control policies can be effectively computed up to prescribed accuracy on the approximate optima of the cost functional.
\end{itemize}

\vspace*{6pt}


\bibliographystyle{newapa}
\bibliography{bib-Osinenko}

\section*{Appendix}

\noindent \emph{Smooth approximations}
\vspace*{6pt}

For the sake of completeness, some technical details
of smooth approximation are discussed in this appendix. First, consider
the following real--valued non--analytic $\mathcal{C}^{\infty}$--function
on $\mathbb{R}$:

\[
\sigma(t)=\left\{ \begin{array}{ll}
\mathrm{e}^{-\nicefrac{1}{t^{2}}} & t>0,\\
0 & t\le0.
\end{array}\right.
\]

Define a $\mathcal{C}^{\infty}$ bump function $\vartheta:\mathbb{R}^{n}\rightarrow\mathbb{R},n\in\mathbb{N}$
by $\vartheta(x):=a\cdot\sigma\left(1-\|x\|^{2}\right)$ where 

\[
a:=\left(\int\limits _{\mathbb{R}^{n}}\sigma\left(1-\|x\|^{2}\right)dx\right)^{-1}.
\]

Then, the support of $\vartheta$ lies within the unit ball $\mathcal{B}(0,1)$,
i. e.\\
$\left\{ x\in\mathbb{R}^{n}:\vartheta(x)>0\right\} \subset\mathcal{B}(0,1)$.
Clearly, $\forall x\in\mathbb{R}^{n},\vartheta(x)\ge0$ and $\int_{\mathbb{R}^{n}}\vartheta(x)dx=1$.
Now, let $\vartheta_{k}(x):=k^{n}\vartheta(kx)$ for $k\in\mathbb{N}$.
It follows that $\int_{\mathbb{R}^{n}}\vartheta_{k}(x)dx=k^{n}\int_{\mathbb{R}^{n}}\frac{1}{k^{n}}\vartheta(kx)d(kx)=1$.
Let $f:\mathbb{R}^{n}\rightarrow\mathbb{R}$ be an $L$--Lipschitz
function. It may be assumed that $L=1$ without loss of generality. Clearly,
$f$ is locally integrable, i. e. integrable on any compact subset
of $\mathbb{R}^{n}$ since it is Lipschitz continuous on this subset.
Since $\vartheta_{k}$ is compactly supported, define a $\mathcal{C}^{\infty}$
function $f_{k}$ by convolution as follows:

\begin{align*}
	f_{k}(x)=\int\limits _{\mathbb{R}^{n}}f(\chi)\vartheta_{k}(x-\chi)d\chi=\int\limits _{\mathbb{R}^{n}}f(x-\chi)\vartheta_{k}(\chi)d\chi=\\
	k^{n}\int\limits _{\mathcal{B}\left(0,\nicefrac{1}{k}\right)}f(x-\chi)\vartheta(k\chi)d\chi=\int\limits _{\mathcal{B}\left(0,1\right)}f\left(x-\dfrac{\chi}{k}\right)\vartheta(\chi)d\chi.
\end{align*}

It follows that

\begin{align*}
	f_{k}(x)-f_{k}(y)=\int\limits _{\mathbb{R}^{n}}\big(f(x-\chi)-f(y-\chi)\big)\vartheta_{k}(\chi)d\chi.
\end{align*}

Since $\forall\chi\in\mathbb{R}^{n},\big|f(x-\chi)-f(y-\chi)\big|\vartheta_{k}(\chi)$$\leq\|x-y\|\vartheta_{k}(\chi)$,
it holds that:

\begin{align*}
	\big|f_{k}(x)-f_{k}(y)\big|\le\int\limits _{\mathbb{R}^{n}}\big|f(x-\chi)-f(y-\chi)\big|\vartheta_{k}(\chi)d\chi\le\\
	\|x-y\|\int\limits _{\mathbb{R}^{n}}\vartheta_{k}(\chi)d\chi=\|x-y\|.
\end{align*}

Finally, since $f(x)=f\left(x\right)\cdot1=f\left(x\right)\cdot\int\limits _{\mathcal{B}(0,1)}\vartheta(\chi)d\chi=$$\int\limits _{\mathcal{B}(0,1)}f\left(x\right)\vartheta(\chi)d\chi$,
it follows that

\begin{align*}
	\big|f_{k}(x)-f(x)\big|=\bigg|\int\limits _{\mathcal{B}(0,1)}f\left(x-\dfrac{\chi}{k}\right)\vartheta(\chi)d\chi-f(x)\bigg|=\\
	\bigg|\int\limits _{\mathcal{B}(0,1)}f\left(x-\dfrac{\chi}{k}\right)\vartheta(\chi)d\chi-\int\limits _{\mathcal{B}(0,1)}f\left(x\right)\vartheta(\chi)d\chi\bigg|=\\
	\bigg|\int\limits _{\mathcal{B}(0,1)}\left(f\left(x-\dfrac{\chi}{k}\right)-f(x)\right)\vartheta(\chi)d\chi\bigg|\leq\\
	\sup_{\chi\in\mathcal{B}(0,1)}\Big|\Big|\dfrac{\chi}{k}\Big|\Big|\int\limits _{\mathcal{B}\left(0,1\right)}f\vartheta(\chi)d\chi=\dfrac{1}{k}.
\end{align*}

\noindent \emph{Brouwerian counter-examples}
\vspace*{6pt}

The two examples in Section \ref{sec:intro}, also called Brouwerian counter-examples, demonstrate the computational inability of finding exact optimal control actions and policies in general. The possible computational problems with addressing optimal control may be related to the so-called principles of omniscience \citep[p.~11]{Bishop1985-constr-analysis}. One of them, called \emph{limited principle of omniscience} (LPO) states: 

\begin{defn}
	(LPO) For any sequence binary $\{a_i\}_i$, the following holds: either $a_i = 0$ for all $i$, or there is a $k$ with $a_k = 1$.
\end{defn}

LPO might be related to the inability of a computer to perform an unbounded search and decide exactly whether a given real number is non-zero or exactly zero which may be in turn related to the Turing's Halting problem \citep{Turing1936-Halting-problem}. Let $\{ a_i \}_i$ be a binary sequence  with at least one $1$ at some place which is not a priori known. Let the numbers $b, c$ be defined as follows:
\[
b = \frac{1}{4} \sum_{i=0}^{\infty} \frac{1}{i+1} a_{2i+1}, c = \frac{1}{4} \sum_{i=1}^{\infty} \frac{1}{i+1} a_{2i}.
\] 

If it were known that $b=0$, then one could deduce that all the odd entries of $\{a_i\}$ are zero and, therefore, since one even entry must be $1$, there exists an index $2N$ such that $a_{2N}=1$, i. e.,
\[
b=0 \; \Rightarrow \; \forall i, \; a_{2i+1} =0 \; \Rightarrow \; \exists N. \; a_{2N}=1.
\]

Similarly, if $c=0$, then some even entry of $\{a_i\}$ must be $1$, i. e.,
\[
c=0 \; \Rightarrow \; \forall i, \; a_{2i} =0 \; \Rightarrow \; \exists N. \; a_{2N+1}=1.
\]

In this minimalistic scenario, the appearance of a $1$ can be described by the condition $b=0 \lor c=0$ which implies LPO.  Consider, for instance, the case of the cost function $J(u) = \min \{ u^2 + b , (u-1)^2+c \}$. If an optimal control action $u^*$ could be computed exactly, such that $J(u^*)= \min J$, then either $u^* \geq \frac{1}{3}$ or $u^* \leq \frac{2}{3}$. If  $u^* \geq \frac{1}{3}$, then
\[
(u^*)^2 +b \geq \frac{1}{9}+b > 0
\]

whence $(u^* -1)^2 +c =0 \Rightarrow c=0$. If $u^* \leq \frac{2}{3}$, then
\[
(u^*-1)^2 +c \geq \frac{1}{9}+c > 0
\]
whence $(u^*)^2 +b =0 \Rightarrow b=0$. Therefore, equivalence to LPO is shown.

\end{document}